\documentclass[10pt, a4paper]{article} 

\usepackage[utf8]{inputenc} 

\usepackage{graphicx} 
\usepackage[english]{babel}
\usepackage{booktabs} 
\usepackage{array} 
\usepackage{caption}

\usepackage{amsfonts,amsmath,amssymb,bbm,amsthm,amsbsy}
\usepackage[all]{xy}
\SelectTips{cm}{}

\usepackage[pagebackref,
    ,pdfborder={0 0 0}
  ,urlcolor=black,a4paper,hypertexnames=false]{hyperref}

\usepackage{color}
\usepackage{pdfcolmk}

\makeatletter
\def\blfootnote{\xdef\@thefnmark{}\@footnotetext}
\makeatother
\date{\today%
    \protect\blfootnote{\copyright{\ N.~Heuer, \today}. 
    }}

\usepackage[nouppercase]{scrpage2}

\automark[subsection]{section}
\usepackage{ifthen}
\deftripstyle{myheadings}
             {\ifthenelse{\isodd{\value{page}}}{\leftmark}{\pagemark}}
             {}
             {\ifthenelse{\isodd{\value{page}}}{\pagemark}{\rightmark}}%
             {}{}{}%
\def\varepsilon{\epsilon}

\def\phi{\varphi}

\newtheorem{thm}{Theorem}[section]
\newtheorem{lemma}[thm]{Lemma}
\newtheorem{prop}[thm]{Proposition}
\newtheorem{corr}[thm]{Corollary}
\newtheorem{claim}[thm]{Claim}

\newtheorem*{decision*}{Problem}

\theoremstyle{remark}
\newtheorem{rmk}[thm]{Remark}

\theoremstyle{definition}
\newtheorem{defn}[thm]{Definition}
\newtheorem*{defn*}{Definition}
\newtheorem{exmp}[thm]{Example}

\theoremstyle{theorem}
\newtheorem{theorem}{Theorem}

\usepackage{tikz}
\usepackage{tikz-cd}

\usepackage{IEEEtrantools}

\newenvironment{equ*}[1]{\begin{IEEEeqnarray*}{#1}}{\end{IEEEeqnarray*}}

\newtheorem{manualtheoreminner}{Theorem}
\newenvironment{manualtheorem}[1]{%
  \renewcommand\themanualtheoreminner{#1}%
  \manualtheoreminner
}{\endmanualtheoreminner}

\newcommand{\Z}{\mathbb{Z}}
\newcommand{\N}{\mathbb{N}}

\newcommand{\att}{\texttt{a}}
\newcommand{\btt}{\texttt{b}}
\newcommand{\ctt}{\texttt{c}}
\newcommand{\dtt}{\texttt{d}}

\newcommand{\xtt}{\texttt{x}}
\newcommand{\ytt}{\texttt{y}}

\newcommand{\at}{\texttt{a}}
\newcommand{\bt}{\texttt{b}}
\newcommand{\xt}{\texttt{x}}

\newcommand{\wtt}{\texttt{w}}\newcommand{\vtt}{\texttt{v}}

\newcommand{\Vcl}{\mathcal{V}}
\newcommand{\Wcl}{\mathcal{W}}

\newcommand{\orb}{\mathrm{orb}}
\newcommand{\dbi}{\mathrm{d}_{\texttt{bi}}}
\newcommand{\dcbi}{\mathrm{d}_{\texttt{cbi}}}
\newcommand{\NP}{\mathrm{NP}}

\newcommand{\cl}{\mathrm{cl}}
\newcommand{\scl}{\mathrm{scl}}
\newcommand{\Acl}{\mathcal{A}}
\newcommand{\Icl}{\mathcal{I}}
\newcommand{\Ocl}{\mathcal{O}}
\newcommand{\Bcl}{\mathcal{B}}

\newcommand{\col}{\colon}

\title{Computing commutator length is hard}

\author{Nicolaus Heuer} 

\begin{document}

\maketitle

\begin{abstract}
The commutator length $\cl_G(g)$ of an element $g \in [G,G]$ in the commutator subgroup of a group $G$ is the least number of commutators needed to express $g$ as their product.

If $G$ is a non-abelian free groups, then given an integer $n \in \N$ and an element $g \in [G,G]$ the decision problem which determines if $\cl_G(g) \leq n$ is NP-complete.
Thus, unless P=NP, there is no algorithm that computes $\cl_G(g)$ in polynomial time in terms of $|g|$, the wordlength of $g$. 
This statement remains true for groups which have a retract to a non-abelian free group, such as non-abelian right-angled Artin groups.

We will show these statements by relating commutator length to the \emph{cyclic block interchange distance} of words, which we also show to be NP-complete.
\end{abstract}


\section{Introduction}

Let $G$ be a group and let $[G,G]$ be its commutator subgroup.
For an element $g \in [G,G]$, the \emph{commutator length of $g$ ($\cl_G(g)$)} is defined as
$$
\cl_G(g) = \min \{ k \in \N \mid \exists_{x_1, \ldots, x_k, y_1, \ldots, y_k \in G} : g = [x_1, y_1] \cdots [x_k, y_k] \}.
$$
Commutator length has geometric meaning: 
Let $X$ be a pointed topological space with fundamental group $G$ and let $\gamma$ be a loop representing an element $g \in [G,G]$ in the commutator subgroup of $G$.
Then $\cl_G(g)$ is the smallest genus of a surface $\Sigma$, which maps to $X$ such that $\gamma$ factors through the induced map $\partial \Sigma \to X$.

There are several algorithms to compute $\cl_G(g)$ for a non-abelian free group $G$ \cite{GT, Culler, Bardakov, Ivanov-Fialkovski}. All these algorithms have exponential running time.
In contrast Calegari \cite{scl_rational} provided an algorithm that computes \emph{stable} commutator length in polynomial time, see Section \ref{subsec:scl}.
We thus consider the following decision problem:
\begin{defn*}[CL-$G$]
Let $G$ be a group with generating set $S$, let $\tilde{g} \in F(S)$ be a word representing an element $g \in [G,G]$ and let $k \in \N$. Then the decision problem which determines if $\cl_G(g) \leq k$ is called CL-$G$. The input has size $|\tilde{g}| + k$.
\end{defn*}

The aim of this article is to show the following result.
\begin{theorem} \label{theorem:cl np complete}
Let $G$ be a non-abelian free group. Then the decision problem CL-$G$ is NP-complete.
Suppose that $G$ has a retract to a non-abelian free group and is generated by a set $S$. Then unless P=NP, there is no polynomial time algorithm that, given an element $\tilde{g} \in F(S)$ which represents an element $g \in [G,G]$, computes $\cl_G(g)$ in polynomial time in $|\tilde{g}|_S$.
\end{theorem}

We note that the decision problem which determines if $\cl_G(g) \leq k$ for a fixed integer $k$ is in P; see Proposition \ref{prop:cl for fixed length}. However, the same problem for \emph{chains} (see Section \ref{subsec:cl}) is NP-complete, even if $k=0$; see \cite{SolvProblemFreeGroup} and Theorem \ref{thm:cl chains np compl}.

\subsection*{Method}
There are several algorithms to compute commutator length in non-abelian free groups $G$ (see Section \ref{subsec:algorithms for computing cl}) from which we see that CL-$G$ is in $\NP$.

To show that CL-$G$ is NP-hard, we relate commutator length in free groups to a certain distance function $\dcbi$ between positive words.

 For an alphabet $\Acl$, let $\Acl^+$ be the set of \emph{positive} words (i.e.\ without inverse letters) in $\Acl$.
 For two words $v, w \in \Acl^+$ we say that $w$ is a \emph{cyclic block interchange} of $v$ if there are cyclic permutations $v'$ (resp. $w'$) of $v$ (resp. $w$) such that there are words $w_1, w_2, w_3, w_4 \in \Acl^+ \cup \{ \emptyset \}$ with $v' = w_1 w_2 w_3 w_4$ and $w' = w_1 w_4 w_3 w_2$.

We say that two words $v,w \in \Acl^+$ are \emph{related}, if $v$ and $w$ contain the same number of each letter of $\Acl$. For two such words we 
define $\dcbi(v,w)$ as follows. If $w$ is a cyclic permutation of $v$ we set $\dcbi(v,w)=0$. Else, let $\dcbi(v,w)$ be the smallest integer $k$ such that there is a sequence $z^0, \ldots, z^k$ of positive words with $z^0 = v$, $z^k = w$ and where each $z^i$ is a cyclic block interchange of  $z^{i-1}$. The pseudometric $\dcbi$ between related words of $\Acl^+$ is called the \emph{cyclic block interchange distance}; see Example \ref{exmp:cbi}.
We will consider the following decision problem:
\begin{defn*}[CBI-$\Acl$]
Let $w, v \in \Acl^+$ be two related words and let $k \in \N$. Then the decision problem which determines if $\dcbi(v,w) \leq k$ is called CBI-$\Acl$. The size of the input is $|v|+|w|$.
\end{defn*}
We will first show how commutator length is related to cyclic block interchange distance, using a method similar to \cite{Ivanov-Fialkovski}.
\begin{theorem} \label{theorem: cl may be used to compute cl}
Let $\Acl$ be an alphabet, let $F(\Acl)$ be the free group on $\Acl$ and let $\cl_{F(\Acl)}$, $\dcbi$, CL-$F(\Acl)$ and and CBI-$\Acl$ be as above. Then
\begin{itemize}
\item[(i)] if $v,w \in \Acl^+$ are related then
$$
\dcbi(v,w) = \cl_{F(\Acl)}(v + w^{-1}) \mbox{, and}
$$
\item[(ii)] there is a polynomial-time reduction from CBI-$\Acl$ to CL-$F(\Acl)$.
\end{itemize}
\end{theorem}

Several types of interchange distances of words have been previously studied due to their connections to DNA mutations; see Remark \ref{rmk:variants of cyclic block interchange} and \ref{rmk:dna of bacteria}.
Using a polynomial reduction of $3$-PARTITION to CBI-$\Acl$ similar to \cite{rev_transp} we show:
\begin{theorem} \label{theorem:cbi np complete}
If $|\Acl|>2$ then decision problem CBI-$\Acl$ is $\NP$-complete.
\end{theorem}
Using Theorems \ref{theorem: cl may be used to compute cl} and \ref{theorem:cbi np complete} we may conclude Theorem \ref{theorem:cl np complete} in Section \ref{sec:proof of thm a}.

\subsection*{Organization of this article}
In Section \ref{sec:preliminaries} we will recall well known results on commutator length and related topics, including a survey of the algorithms to compute commutator length in Section \ref{subsec:algorithms for computing cl}.
In Section \ref{sec:cbi cl} we will prove Theorem \ref{theorem: cl may be used to compute cl} which shows how commutator length may be used to compute the cyclic block interchange distance.
In Section \ref{sec:cbi} we will prove Theorem \ref{theorem:cbi np complete} which shows that for any alphabet $\Acl$ the decision problem CBI-$\Acl$ is NP-hard.
In Section \ref{sec:proof of thm a} we prove Theorem \ref{theorem:cl np complete}.
 In the Appendix (Section \ref{sec:appendix}) we will provide a MATLAB code needed to check some cases of Lemma \ref{lemma: nu}.

\subsection*{Acknowledgments}
I would like to thank Martin Bridson for many very helpful discussions on this problem. This project started in joint discussion with Michał Marcinkowski and Clara L{\"o}h at the University of Regensburg. I would like to thank both of them for helpful discussions and the university (supported by the SFB 1085 Higher Invariants) for their hospitality. I would like to thank Clara L{\"o}h particularly for many helpful comments on a previous version of this paper.
Finally, I would like to thank Joana Guiro for her helpful insights into the genes of bacteria.

\section{Preliminaries} \label{sec:preliminaries}

In this section we will survey well known results relating to commutator length.
We will discuss general properties of commutator length and define commutator length on chains in Section \ref{subsec:cl}. In Section \ref{subsec:scl} we briefly discuss \emph{stable} commutator length. In Section \ref{subsec:algorithms for computing cl} we recall the algorithms available to compute commutator length in free groups. In Section \ref{subsec:cl on chains np complete} we recall results of \cite{SolvProblemFreeGroup} which show that the problem which decides if a \emph{chain} has commutator length $0$ is NP-complete.

\subsection{Commutator length} \label{subsec:cl}

Let $G$ be a group and let $[G,G]$ be its commutator subgroup. For an element $g \in [G,G]$ the \emph{commutator length of $g$} ($\cl_G(g)$) is defined as
$$
\cl_G(g) = \min \{ k \in \N \mid \exists_{x_1, \ldots, x_k, y_1, \ldots, y_k \in G} : g = [x_1, y_1] \cdots [x_k, y_k] \},
$$
where for two elements $x,y \in G$, $[x,y] = x y x^{-1} y^{-1}$ denotes the commutator bracket. It is easy to see that commutator length is invariant under conjugation.

A formal sum $g_1+ \cdots + g_n$ of elements in $G$ will be called a \emph{chain}, if $g_1 \cdots g_n \in [G,G]$. The set of all chains of $G$ is denoted by $B_1(G)$.
For a chain $g_1+ \cdots + g_n$ as above we set
$$
\cl_G(g_1 + \cdots + g_n) = \min \{ \cl_G(t_1 g_1 t_1^{-1} \cdots t_n g_n t_n^{-1}) \mid t_1, \ldots, t_n \in G \}.
$$

\begin{prop}[Properties of commutator length] \label{prop:formulas for cl}
Let $G$ be a group and let $\cl_G$ be as above. Then we have the following identities.
\begin{itemize}
\item If $\Phi \col G \to H$ is a homomorphism and $g \in [G,G]$, then $\cl_{G}(g) \geq \cl_{H}(\Phi(g))$. If $H < G$ is a retract of $G$ and $g \in [H,H]
$ then $\cl_{H}(g) = \cl_{G}(g)$.
\item 
For any $g,h \in G$, we have $\cl_G(h g h^{-1}) = \cl_G(g)$. Thus, if $G$ is a non-abelian free group, then commutator length is invariant under a cyclic permutation of the letters of $g$.
\item For two chains $c_1, c_2 \in B_1(G)$ we have $\cl_G(c_1 + c_2) \leq \cl_G(c_1) + \cl_G(c_2)$.
\item For an element $g \in G$ and a chain $c \in B_1(G)$ we have $\cl_G(g + g^{-1} + c) = \cl_G(c)$.
\end{itemize}
\end{prop}
All of these basic statements may be found in \cite[Section 2]{Calegari}.

\subsection{Stable commutator length} \label{subsec:scl}

For an element $g \in [G,G]$ we define the \emph{stable commutator length} of $g$ in $G$ by setting
$$
\scl_G(g) = \lim_{m \to \infty} \frac{\cl_G(g^m)}{m}
$$
and for a chain $g_1 + \cdots + g_n \in B_1(G)$ we define 
$$
\scl_G(g_1 + \cdots + g_n) = \lim_{m \to \infty} \frac{\cl_G(g_1^m + \cdots + g_n^m)}{m}.
$$

This invariant has seen a vast development in recent years, most prominently by Calegari and others
\cite{Calegari, heuer_scl_rp, heuer_scl_simvol, heuer_chen_spectral}.

Computing $\scl_G$ may seem like a harder problem than computing $\cl_G$. However, Calegari showed that if $G$ is a non-abelian free group, then $\scl_G(g)$ may be computed in polynomial time in the wordlength of $G$ \cite{scl_rational}.

On the other hand Brantner \cite{Lukas} showed that computing $\scl_G$ is NP-complete on free groups if powers of elements in the free group have logarithmic size.

\begin{exmp} \label{exmp:scl}
Using the algorithm of Calegari \cite{Calegari}, one may compute that $\scl_{F(\att, \btt)}([\att, \btt]) = 1/2$ and that $\scl_{F(\att, \btt)} (\att \btt + \att^{-1} + \btt^{-1}) = 1/2$.
\end{exmp}
There is always a gap of $1/2$ for $\scl$ on free groups \cite{DH, heuer_raags}.

\subsection{Algorithms for computing commutator length in free groups} \label{subsec:algorithms for computing cl}

In this section we survey algorithms to compute commutator length in non-abelian free groups.
The first algorithm to compute commutator length was done by Goldstein and Turner in \cite{GT}. 
Later Culler gave a more geometric algorithm \cite{Culler}, using surface maps. Based on this algorithm, Bardakov gave a purely algebraic algorithm \cite{Bardakov}. Later, Fialkovski and Ivanov \cite{Ivanov-Fialkovski} gave a combinatorial algorithm to compute commutator length, using interchanges of subwords, which we will adapt for Section \ref{sec:cbi cl}.

We state the algorithms of Bardakov and Fialkovski--Ivanov in Sections \ref{subsec:bardakov} and \ref{subsec:fi} respectively. In Corollary \ref{corr:cl is in np} we deduce that CL is in NP and in Proposition \ref{prop:cl for fixed length} we show that the decision problem CL-$G$ \emph{is} in P for a fixed $k$.

We illustrate the difficulty to compute commutator length by the following example:
\begin{exmp}[\cite{DH}] \label{exmp:cl}
We have that $\cl_{F(\att,\btt)}([\att,\btt]^n) = \lfloor \frac{n-1}{2} \rfloor + 1$. Similarly, we see that $\cl_{F(\att,\btt)}(\att^n \btt^n + (\att \btt)^{-n}) = \lfloor \frac{n}{2}\rfloor$.
\end{exmp}
Compare this example with Example \ref{exmp:scl}.

\subsubsection{Bardakov's algorithm} \label{subsec:bardakov}

We describe the algorithm of Bardakov to compute commutator length and extend his algorithm to chains.

Let $\Acl$ be an alphabet and let $F=F(\Acl)$ be the free group on $\Acl$.
An element $x \in F$ will be called \emph{letter} if either $x \in \Acl$ or $x^{-1} \in \Acl$.
To stress that an element is a letter, we will write it in code-font e.g. $\at, \bt, \ldots$.

Let $w = w_1 + \ldots + w_k \in B_1(F)$ be a chain where each $w_j$ is cyclically reduced, has wordlength $n_j$, and $w_j = \xtt_{1,j} \cdots \xtt_{n_j,j}$. We call $\Icl_w = \{ (i,j) \mid 1 \leq i \leq n_j, 1 \leq j \leq k \}$ the \emph{index set} of the chain $w$.

Following the notation of \cite{Bardakov}, a \emph{pairing of $w$} is an map $\pi \col \Icl_w \to \Icl_w$ such that
\begin{itemize}
\item  $\pi$ is a fixed point free involution, that is $\pi$ does not have a fixed point and $\pi(\pi(i,j)) = (i,j)$ for all $(i,j) \in \Icl_w$, and

\item for any $(i,j) \in \Icl_w$ we have that $\xt_{\pi(i,j)} = \xt_{i,j}^{-1}$.
\end{itemize}

The set of all pairings of $w$ is denoted by $\Pi_w$.
We define $\sigma \col \Icl_w \to \Icl_w$ via
\[
\sigma \col (i,j) \mapsto \begin{cases} (i+1,j) &  \text{ if } i<n_j \\
(1,j) & \text{ if} i = n_j
\end{cases}
\]
i.e. $\sigma$ cycles through each index of the words.

\begin{thm}[\cite{Bardakov}] \label{thm:bardakov}
For a chain $w = w_1 + \cdots + w_k \in \textrm{B}_1(F)$ we have
\[
\cl_F(w) = \frac{|w|}{4} - \frac{o}{2} + \frac{2-k}{2}
\]
where $|w| = \sum_{j=1}^k |w_j|$ is the total word length of $w$ and $o = \max \{ \orb(\sigma \pi) \mid \pi \in \Pi_w \}$.
\end{thm}
Here and throughout the paper, $\orb(\sigma \pi)$ will denote the number of orbits of the permutation $\sigma \pi \col \Icl_w \to \Icl_w$.

We note that Bardakov only proved this for single commutators in a two generator free group. However, his proof immediately generalizes to Theorem~\ref{thm:bardakov}.

Note that, given a chain $w$ and an integer $n \in \N$, we may verify that $\cl_F(w) \leq n$ by providing an appropriate pairing $\pi \in \Pi_w$. Computing the number of orbits can be done in linear time in $|w|$. Thus we may verify $\cl_F(w)$ in polynomial time. Thus we see:

\begin{corr} \label{corr:cl is in np}
For any non-abelian free group $F$, CL-$F$ is in $\NP$, where the input size is measured in the wordlength of a free basis of $F$.
\end{corr}

\subsubsection{Fialkovski-Ivanov's algorithm} \label{subsec:fi} 

In 2015, Fialkovski and Ivanov gave an algorithm to compute commutator length and an explicit presentation in terms of commutators, which is based on interchanges of subwords.

\begin{thm} [\protect{\cite[Theorem 5.1]{Ivanov-Fialkovski}}] \label{thm:fi}
Let $F = F(\Acl)$ be a free group and let $w \in {F,F}$ be cyclically reduced. Then there is a presentation of $w$ without cancellations such that
$$
w = w_1 \xtt^{-1} w_2 \ytt^{-1} w_3 \xtt w_4 \ytt w_5
$$
with $\xtt, \ytt \in \Acl^{\pm}$ such that
$$
w = [w_1 w_4 w_3 \xtt w_1^{-1}, w_1 w_4 \ytt w_2^{-1} w_3^{-1} w_4^{-1} w_1^{-1}] w_1 w_4 w_3 w_2 w_5
$$
and
$$
\cl_F(w_1 w_4 w_3 w_2 w_5) = \cl_F(w) - 1.
$$
\end{thm}
Observe that in the above theorem, the word $w_1 w_4 w_3 w_2 w_5$ is obtained from the word $w$ by exchanging the subwords $w_2$ and $\xtt w_4 \ytt$ with each other.
Theorem \ref{thm:fi} gives an immediate algorithm to compute commutator length.
This has the following immediate consequence:

\begin{prop}[CL for fixed integer is polynomial] \label{prop:cl for fixed length}
Fix an integer $k \in \N$ and let $F$ be a non-abelian free group.
Then the decision problem that decides if an element $g \in [F,F]$ satisfies $\cl_F(g) \leq k$ is in P.
\end{prop}

\begin{proof}
Immediate from the previous theorem. 
\end{proof}
Note however, that a crude estimation of the running time is of the order $\Ocl(|g|^{4 k} )$. Thus this does not yield a polynomial time algorithm to compute commutator length for general elements.

\subsection{Computing commutator length on chains is NP complete} \label{subsec:cl on chains np complete}

Proposition \ref{prop:cl for fixed length} showed that for a non-abelian free group $G$ the decision problem which determines if $\cl_G(g) \leq k$ for a fixed integer $k$ is in $P$.
In this section we show that this is not the case for chains.
This restates a result of 
Kharlampovich, Lys\"{e}nok, Myasnikov and Touikan \cite{SolvProblemFreeGroup}:

\begin{thm} [\cite{SolvProblemFreeGroup}] \label{thm:cl chains np compl}
Let $F$ be a non-abelian free group. Then the decision problem which decides if a chain $c \in B_1(F)$ satisfies $\cl_F(c)=0$ is NP-complete.
\end{thm}

\begin{proof}
We outline the strategy descirbed in \cite{SolvProblemFreeGroup}.
By Bardakov's algorithm it is clear that this problem is in NP.

To show that it is NP-hard, we will reduce the \emph{exact bin packing problem} to it:

\begin{defn}[EBP]
Given a $k$-tuple $(n_1, \ldots, n_k)$ and positive integers $B,N$ which are uniformly polynomial in $k$. 
Then the decision problem which decides if there is a partition of $\{ 1, \ldots, k \}$ into $N$ subsets
$$
\{ 1, \ldots, k \} = S_1 \sqcup \cdots \sqcup S_N
$$
such that for each $i = 1, \ldots, N$ we have
$$
\sum_{j \in S_j} n_j = B
$$
is called \emph{exact bin packing} and denoted by EBP.  The input size is $\left( \sum_{i=1}^k n_i \right) + N + B$. As $N$ and $B$ were polynomial in $k$ we see that the whole input size is polynomial in $k$.
\end{defn}
This problem is known to be NP-complete \cite[p. 226]{comp_intrac}.
Given an instance $(n_1, \ldots, n_k)$, $N$ and $B$ of EBP, define the chain $c \in B_1(F(\{ \att, \btt \}))$ as
$$
c = [\att, \btt^{n_1}] + \cdots [\att, \btt^{n_k}] + [\att^N, \btt^{B}]^{-1}.
$$
By \cite[Theorem 3.11]{SolvProblemFreeGroup}
we have that there is a solution to this instance of EBP if and only if $\cl_{F(\{ \att, \btt \})}(c) = 0$. Thus we may reduce EBP to the problem of computing the commutator length of chains and thus this problem is NP-hard.
\end{proof}

\section{CBI and CL} \label{sec:cbi cl}

The aim of this section is to prove Theorem \ref{theorem: cl may be used to compute cl}.
We first recall cyclic block interchange distance $\dcbi$ and highlight its connection to other fields in Section \ref{subsec:cyclic block interchange distance}.
In Section \ref{subsec:cbi as cl} we prove part (i) of Theorem \ref{theorem: cl may be used to compute cl}, which shows how commutator length can be used to compute $\dcbi$.
In Section \ref{subsec: poly redcution of cbi to cl} we  prove part (ii) of Theorem \ref{theorem: cl may be used to compute cl}, which shows that for any alphabet $\Acl$, CBI-$\Acl$ may be reduced to CL-$F(\Acl)$ in polynomial time.

\subsection{Cyclic block interchange distance} \label{subsec:cyclic block interchange distance}
Let $\Acl$ be an alphabet and let $\Acl^+$ be the set of positive words in $\Acl$, i.e.\ words in $\Acl$ without inverses.
For two words $v, w \in \Acl^+$ we say that $w$ is a \emph{cyclic block interchange} of $v$ if there are cyclic permutations $v'$ (resp. $w'$) of $v$ (res of $w$) such that there are words $w_1, w_2, w_3, w_4 \in \Acl^+ \cup \{ \emptyset \}$ with 
\begin{eqnarray*}
v' &=& w_1 w_2 w_3 w_4, \mbox{ and}\\
w' &=& w_1 w_4 w_3 w_2.
\end{eqnarray*}
There is a more geometric way to think about cyclic block interchange. If $S_v$ and $S_w$ are circles labeled by the words $v$ and $w$ then $w$ is a cyclic block interchange of $v$ if one may obtain $S_w$ by exchanging two subsegments of $S_v$.

We say that two words $v,w \in \Acl^+$ are \emph{related}, if both $v$ and $w$ contain the same number of each letter of $\Acl$. It is easy to see that for two related words $v$ and $w$ there is a sequence of cyclic block interchanges $z^0, \ldots, z^k$ such that $z^0 = v$, $z^k = w$ and such that each $z^i$ is a cyclic block interchange of $z^{i-1}$.
For two related words $v, w$ as above we set $\dcbi(v,w)=0$ if $w$ is a cyclic permutation of $v$. Else, $\dcbi(v,w)$ is the smallest $k$ such that there is a sequence $z^0, \ldots, z^k$ as above.

\begin{exmp} \label{exmp:cbi}
In the alphabet $\Acl = \{ \att, \btt \}$ consider the two related words 
$v = \att \btt \att \btt \att \btt$ and $w = \att \att \att \btt \btt \btt$.
We see that $w$ is obtained from a cyclic block interchange: Indeed, for $v' = \btt \att \btt \att \btt \att$ and $w' = \btt \att \att \att \btt \btt$ note that $v'$ (resp. $w'$) is a cyclic permutation of $v$ (resp. $w$) and that $v' = w_1 w_2 w_3 w_4$ and 
$w' = w_1 w_4 w_3 w_2$ 
 with $w_1 = \btt \att$, $w_2 = \btt$, $w_3 = \att \btt$, $w_4 = \att$.
Thus, as $v$ and $w$ are not cyclic permutations of each other, we conclude that $\dcbi(v,w)=1$.
More generally we may define 
$v_n = (\att \btt)^n$ and $w_n = \att^n \btt^n$. Then, using Theorem \ref{theorem: cl may be used to compute cl} and Example \ref{exmp:scl} we see that $\dcbi(v_n,w_n) = \lfloor \frac{n}{2} \rfloor$
\end{exmp}

Several variations of cyclic block interchange have been studied in the literature:

\begin{rmk}[Variants of cyclic block interchange] \label{rmk:variants of cyclic block interchange}
The block interchange distance is a variation of \emph{transposition distance} studied in \cite{rev_transp}, which are interchanges of the form $w_1 w_2 w_3 w_4 \mapsto w_1 w_3 w_2 w_4$, and some of our arguments in Section \ref{sec:cbi} are an adaptation of the methods of \cite{rev_transp}.
It is also a variation of \emph{block interchange distance}, which was studied in \cite{Christie}, which are interchanges of the form $w_1 w_2 w_3 w_4 w_5 \mapsto w_1 w_4 w_3 w_2 w_5$.
There are also several other versions of this problem, such as assuming that the words contain only distinct letters, interchanging only letters, reversing the order of the subwords, inserting defects, etc. See \cite{CGR} for a comprehensive survey on this topic.
\end{rmk}

\begin{rmk}[DNA of bacteria] \label{rmk:dna of bacteria}
Those theoretical considerations have an application in molecular biology:
A DNA is a double stranded molecule composed of a sequence of simpler monomeric units called nucleotides. There are four types of nucleotides, typically denoted by letters $\Acl = \{ G, T, A, C \}$.
Thus every DNA is a positive word in the alphabet $\Acl$. Note that in bacteria, the DNA is circular and encodes several genes.

 Subwords of DNA with a biological function are called \emph{genes}. 
Certain genes, called \emph{transposons} have the property that they might change their position within the DNA and switch it with others.
Consider one model for a permutation in which two such substrands of DNA are exchanged.
Let $v$ and $w$ be two (circular) bacterial DNA strands, where $v$ is obtained from $w$ by a sequence of transpositions.
Then $\dcbi(v,w)$ measures the least number of such permutations.

However, the results of this paper shows that this quantity is inherently hard to compute.
\end{rmk}

\subsection{$\dcbi$ via $\cl$} \label{subsec:cbi as cl}

The aim of this section is to show the part $(i)$ of Theorem \ref{theorem: cl may be used to compute cl}:

\begin{manualtheorem}{ \ref{theorem: cl may be used to compute cl} (i)}
If $v, w \in \Acl^+$ are related, then
$$
\dcbi(v,w) = \cl_{F(\Acl)}(v + w^{-1}).
$$
\end{manualtheorem}

For the rest of this section, we fix two related words $v,w \in \Acl^+$ and suppose that $|v|=|w|=n$ for some $n \in \N$.
We write
\begin{eqnarray*}
v &=& \xt_{0^+} \cdots \xt_{(n-1)^+} \mbox{, and} \\
w &=& \xt_{0^-} \cdots \xt_{(n-1)^-}.
\end{eqnarray*}
We write the index set of  of Bardakov (see Section \ref{subsec:bardakov}) in the following way:
\begin{rmk} \label{rmk:index set bardakov}
We set 
$\Icl_{v+w^{-1}} = \Icl^+ \cup \Icl^-$ with $\Icl^\pm = \{ i^\pm; i \in \Z_n \}$, for $\Z_n$ the cyclic group with $n$ elements.
For $x_1, x_2 \in \Icl^+$ with $x_1 = i_1^+$, $x_2 = i_2^+$ we write $x_1 + x_2 = (i_1 + i_2)^+$ where $i_1 + i_2$ is addition in $\Z_n$. Similarly, $x_1 - x_2 = (i_1 - i_2)^+$. Analogously, we may add and subtract two elements of $\Icl^-$.
If $x \in \Icl^+$ with $x = i^+$ then by $x + 1$ we mean the element $(i + 1)^+ \in \Icl^+$. We set $|x|=i$, where $i$ is the integer $0 \leq i \leq n-1$ such that $x = i^+$.

Note that in particular, for two elements $x, y \in \Icl^+$, we do not generally have that $|x-y| = |y - x|$.
\end{rmk}

Every pairing $\pi \col \Icl_{v + w^{-1}} \to \Icl_{v + w^{-1}}$ of $v + w^{-1}$ in the sense of Bardakov has to match up any letter with its inverse. Since $v$ contains only letters of $\Acl$ and $w^{-1}$ contains only inverse letters of $\Acl$ we see that $\pi(\Icl^-) = \Icl^+$ and that $\xt_{\pi(i^+)} = \xt_{i^+}$ for any $i \in \Icl^+$. 
We denote the set of pairings by $\Pi_{v+w^{-1}}$.
As in Section \ref{subsec:bardakov} we define $\sigma \col \Icl_{v + w^{-1}} \to \Icl_{v + w^{-1}}$ via
$$
\sigma(x) = \begin{cases}
x + 1  & \text{if } x \in \Icl^+ \\
x - 1  & \text{if } x \in \Icl^-.
\end{cases}
$$
See Remark \ref{rmk:index set bardakov} for the notation involved in this statement. Bardakov's formula reduces to
$$
\cl_{F(\Acl)}(v + w^{-1}) = \min_{\pi \in \Pi_{v+w^{-1}}} \Big{ \{ } \frac{n}{2} - \frac{\orb(\sigma \pi)}{2} \Big{ \} } 
$$
as $|v + w^{-1}| = |v| + |w| = 2n$.

\subsubsection{Notation} \label{subsubsec:notation}

We call a sequence $(x_1, x_2, \ldots, x_k)$ of elements $x_j \in \Icl^+$ \emph{cyclically ordered} if it is possible to find representatives $i_1, \ldots, i_k \in \Z$ of $x_1, \ldots, x_k$ and an integer $N \in \N$ such that 
$$
N \leq i_1 <  \cdots <  i_k < N + n.
$$
See \ref{rmk:index set bardakov} for the notations involved in this statement.
We say that $y \in \Icl^+$ \emph{lies cyclically between $x_1, x_2 \in \Icl^+$} if $(x_1, y, x_2)$ is cyclically ordered.

For a bijection $\alpha \col \Icl_{v+w^{-1}} \to \Icl_{v+w^{-1}}$ and an element $x \in \Icl_{v+w^{-1}}$ let $o(x)$ be the size of the orbit of $x$ under $\alpha$ and let $\Ocl_\alpha(x)$ be the ordered sequence
$$
\Ocl_\alpha(x) = \left( x, \alpha(x), \ldots, \alpha^{o(x)-1}(x) \right).
$$
If $x \not= y \in \Icl_{v+w^{-1}}$ lies in the orbit of $x$ under $\alpha$ then let $o(x,y)$ be the least positive integer $o \in \N^+$ such that $\alpha^o(x) = y$ and let $\Ocl_\alpha(x,y)$ be the ordered sequence
$$
\Ocl_\alpha(x,y) = \left( x, \alpha(x), \ldots, \alpha^{o(x,y)-1 }(x) \right).
$$
The concatenation between two such sequences is denoted by $"\cdot"$. For example, it is easy to see that if $x \not = y$ is in the orbit of $x$ under $\alpha$ then $\Ocl_\alpha(x) = \Ocl_\alpha(x,y) \cdot \Ocl_\alpha(y,x)$.

Let $y_1,y_2$ be both in the orbit of $x$ under $\alpha$ such that all of $x,y_1,y_2$ are distinct. Then observe that
\[
\Ocl_\alpha(x) = \Ocl_\alpha(x, y_1) \cdot \Ocl_\alpha(y_1, y_2) \cdot \Ocl_\alpha(y_2, x)
\]
if and only if there are integers $0 < i_1 < i_2 < o(x)$ such that $y_1 = \alpha^{i_1}(x)$ and $y_2 = \alpha^{i_2}(x)$.

\subsubsection{Orbits of $\sigma \pi$} \label{subsec:orbits of sigma pi}

Let $\pi \in \Pi_{v+w^{-1}}$ be a pairing  and let $\sigma \col \Icl_{v+w^{-1}} \to \Icl_{v+w^{-1}}$ be as above.
We analyze which type of orbits arise for $\alpha = \sigma \pi$.

\begin{lemma}[orbits of $\alpha$] \label{lemma:special points for alpha}
Let $\pi \in \Pi_{v+w^{-1}}$ be a pairing and let $\alpha = \sigma \pi$. Then either
\begin{enumerate}
\item $\alpha^2(x) = x$ for all $x \in \Icl_{v+w^{-1}}$, or
\item there are elements $i_0^+, i_1^+, i_2^+ \in \Icl^+$ such that $(i_0^+, i_1^+, i_2^+)$ is cyclically ordered and such that
$$
\Ocl_\alpha(i_0^+) = \Ocl_\alpha(i_0^+, i_2^+) \cdot \Ocl_\alpha(i_2^+, i_1^+) \cdot \Ocl_\alpha(i_1^+, i_0^+) \mbox{, or}
$$
\item there are elements $i_0^+, i_1^+, i_2^+, i_3^+ \in \Icl^+$ such that $(i_0^+, i_1^+, i_2^+, i_3^+)$ is cyclically ordered such that $\Ocl_\alpha(i_0^+)$ and $\Ocl_\alpha(i_1^+)$ are distinct orbits with
\begin{align*}
\Ocl_\alpha(i_0^+) &= \Ocl_\alpha(i_0^+, i_2^+) \cdot \Ocl_\alpha(i_2^+, i_0^+) \mbox{, and}\\
\Ocl_\alpha(i_1^+) &= \Ocl_\alpha(i_1^+, i_3^+) \cdot \Ocl_\alpha(i_3^+, i_1^+).
\end{align*}
\end{enumerate}
\end{lemma}

\begin{proof}
If $\alpha^2(x) = x$ for all $x \in \Icl_{v+w^{-1}}$ we are done.
Else observe that $\alpha(\Icl^+) = \Icl^-$ and $\alpha(\Icl^-)=\Icl^+$ and thus $\alpha^2(\Icl^+)=\Icl^+$. We see that if 
$\alpha^2(x) \not = x$ for some $x \in \Icl_{v+w^{-1}}$ then there is some $x \in \Icl^+$ such that $\alpha^2(x) \not = x$.

Choose $i_0^+ \in \Icl^+$ with the property that $\alpha^2(i_0^+) \not = i_0^+$ and 
such that $i_0^+$ satisfies
\[
|\alpha^2(i_0^+)- i_0^+| = \min \{ | \alpha^2(x) - x | \mid \alpha^2(x) \not = x, x \in \Icl^+ \}  = M,
\]
thus $\alpha^2(i_0^+) = (i_0 + M)^+$ and $M < n$. 
See Remark \ref{rmk:index set bardakov} for the notation.

\begin{claim}
Let $K>0$ be the smallest integer such that $\alpha^2(i_0^+ + K) \not = i_0^+ + K$. Then $K < M-1$.
\end{claim}

\begin{proof}

Suppose that $\alpha^2(i_0^+ +k) = i_0^+ + k$ for all $0 \leq k \leq M-1$. 

Note that $M > 1$. Else we have that $\alpha^2(i_0^+) =  i_0^+ + 1$. Thus
\begin{eqnarray*}
\sigma \pi \sigma \pi(i_0^+) &=& i_0^+ + 1 \\
\pi(\sigma(\pi(i_0^+)) &=& i_0^+ \\
\pi( \pi(i_0^+ -1) &=& i_0^+ \\
\pi(i_0^+ - 1) &=& \pi(i_0^+) \\
i_0^+ - 1 &=& i_0^+
\end{eqnarray*}
which is a contradiction.

Thus suppose that $M > 1$.
Then note that if $\alpha^2(i_0^+ +k) = i_0^+ +k$, then
\begin{eqnarray*}
\sigma \pi \sigma \pi(i_0^+ + k) &=& i_0^+ + k \\
\pi(\sigma(\pi(i_0^+ + k))) &=& i_0^+ + (k-1)\\
\sigma(\pi(i_0^+ + k)) &=& \pi(i_0^+ + (k-1)) \\
\pi(i_0^+ + k) &=& \pi(i_0^+ + (k-1)) + 1. \\
\end{eqnarray*}
Inductively we see that for every $0 \leq k \leq M-1$ we have that
$\pi(i_0^+ + k) = \pi(i_0^+) + k$
and in particular 
$$
\pi(i_0^+ + (M-1)) = \pi(i_0^+) + (M-1).
$$
Also we have that
\begin{eqnarray*}
\sigma \pi \sigma \pi(i_0^+) &=& i_0^+ + M \\
\pi(\pi(i_0^+)-1) &=& i_0^+ + (M-1) \\
\pi(i_0^+)-1 &=& \pi(i_0^+ + (M-1),
\end{eqnarray*}
and, by comparing the terms we see that
$$
\pi(i_0^+)-1 = \pi(i_0^+ + (M-1) = \pi(i_0^+) + (M-1)
$$
and hence $M = n$, which is a contradiction. 
\end{proof}

By minimality of $M$ we see that $\alpha^2((i_0+K)^+)$ does not lie cyclically between $i_0^+$ and $\alpha^2(i_0^+)$. There are two cases:
\begin{enumerate}
\item $(i_0+K)^+$ belongs to the orbit of $i_0^+$ under $\alpha$. Then we set
\[
i_2^+ = (i_2+M)^+ \text{ and } i_1^+ = (i_0 + K)^+
\]
then we see that $(i_0^+, i_1^+, i_2^+)$ are cyclically ordered and moreover,
\[
\Ocl_\alpha(i_0^+) = \Ocl_\alpha(i_0^+, i_2^+) \cdot \Ocl_\alpha(i_2^+, i_1^+) \cdot \Ocl_\alpha(i_1^+, i_0^+)
\]
\item $(i_0+K)^+$ does not belong to the orbit of $i_0^+$ under $\alpha$. Then set
\[
i_2^+ = (i_2+M)^+ \text{, } i_1^+ = (i_0 + K)^+ \text{, and } i_3^+ = \alpha(i_1^+)
\]
then again $(i_0^+, i_1^+, i_2^+, i_3^+)$ are cyclically ordered, $\Ocl_\alpha(i_0^+)$ and $\Ocl_\alpha(i_1^+)$ are distinct orbits and
\begin{align*}
\Ocl_\alpha(i_0^+) &= \Ocl_\alpha(i_0^+, i_2^+) \cdot \Ocl_\alpha(i_2^+, i_0^+) \mbox{, and} \\
\Ocl_\alpha(i_1^+) &= \Ocl_\alpha(i_1^+, i_3^+) \cdot \Ocl_\alpha(i_3^+, i_1^+).
\end{align*}
\end{enumerate}
This completes the proof of Lemma \ref{lemma:special points for alpha}.
\end{proof}

\subsubsection{Modifying the orbits by a cyclic block interchange} \label{subsec:modifying the orbits}

 We define two types of maps $\gamma_{i_1, i_2}, \gamma_{i_1,i_2,i_3} \col \Icl_{v+w^{-1}} \to \Icl_{v + w^{-1}}$ which are permutations of the index set $\Icl_{v + w^{-1}}$. These maps will correspond to cyclic block interchanges.
\begin{enumerate}
\item Let $i_1,i_2$ be integers with  $0<i_1<i_2 \leq n-1$.
Then $\gamma_{i_1,i_2} \col \Icl_{v+w^{-1}} \to \Icl_{v+w^{-1}}$ is defined as follows. If $x^- \in \Icl^-$, set $\gamma(x^-) = x^-$. For $x^+ \in \Icl^+$
set 
\[
\gamma_{i_1,i_2}(x^+) = \begin{cases}
x^+ & 					\text{if } 0^+ \leq x^+ < i_1^+ \\
x^+ + i_2 - i_1  &		\text{if } i_1^+ \leq x^+ < i_2^+ \\
x^+ - i_2 + i_1  &		\text{if } i_2^+ \leq x^+ \leq (n-1)^+.
\end{cases}
\]

\item Let $i_1,i_2,i_3$ be integers such that $0 < i_1 < i_2 < i_3 \leq n-1$.
Then $\gamma_{i_1,i_2,i_3} \col \Icl_{v+w^{-1}} \to \Icl_{v+w^{-1}}$ is defined as follows. If $x^- \in \Icl^-$ set $\gamma_{i_1,i_2,i_3}(x^-) = x^-$. For $x^+ \in \Icl^+$ set 
\[
\gamma_{i_1,i_2,i_3}(x^+) = \begin{cases}
x^+ & 					\text{if } 0^+ \leq x^+ < i_1^+ \\
(x^+ + n - i_2  &		\text{if } i_1^+ \leq x^+ < i_2^+ \\
x^+ - i_2 + n-i_3 + i_1 & \text{ if } i_2^+ \leq x^+ < i_3^+ \\
x^+ - i_3 + i_1  &		\text{if } i_3^+ \leq x^+ \leq (n-1)^+.
\end{cases}
\]
\end{enumerate}
We see that if $v = \xtt_{0^+} \cdots \xtt_{{(n-1)}^+}$, then $\gamma_{i_1,i_2}$ corresponds to the block interchange from $v_1 v_2 v_3$ to $v_1 v_3 v_2$ where $v_1 = \xtt_{0^+} \cdots \xtt_{(i_1-1)^+}$, $v_2 = \xtt_{i_1^+} \cdots \xtt_{(i_2-1)^+}$ and $v_3 = \xtt_{i_2^+} \cdots \xtt_{(n-1)^+}$.
Similarly, $\gamma_{i_1,i_2,i_3}$ corresponds to the block interchange from $v = v_1 v_2 v_3 v_4$ to $v_1 v_4 v_3 v_2$ where  $v_1 = \xtt_{0^+} \cdots \xtt_{(i_1-1)^+}$, $v_2 = \xtt_{i_1^+} \cdots \xtt_{(i_2-1)^+}$, 
$v_3 = \xtt_{i_2^+} \cdots \xtt_{(i_3-1)^+}$
 and $v_4 = \xtt_{i_3^+} \cdots \xtt_{(n-1)^+}$.

We observe that if $\gamma$ is one of the above, then if $\pi$ is a pairing, then $\gamma^{-1} \pi \gamma$ is a pairing for $v + w^{-1}$ after applying the corresponding block interchange to $v$.
We will write $\alpha_\gamma = \sigma \gamma^{-1} \pi \gamma$ and note that this is the function $\alpha$ obtained after the corresponding block interchange.

\begin{lemma}[Modifying the orbits by a block interchange] \label{lemma:modifying orbits}
We have the following:
\begin{enumerate}
\item \label{prop_case:orbits_case1} If there are $i_1^+, i_2^+$ with $(0^+, i_1^+, i_2^+)$ cyclically ordered, such that 
\[
\Ocl_\alpha(0^+) = \Ocl_\alpha(0^+, i_2^+) \cdot \Ocl_\alpha(i_2^+, i_1^+) \cdot \Ocl_\alpha(i_1^+, 0^+).
\] 
Then we have that $\orb(\alpha_{\gamma_{i_1,i_2}}) = \orb(\alpha) + 2$.
\item \label{prop_case:orbits_case2} If there are $i_1^+, i_2^+, i_3^+$ with $(0^+, i_1^+, i_2^+, i_3^+)$ cyclically ordered, such that $\Ocl_\alpha(0^+)$ and $\Ocl_\alpha(i_1^+)$ are distinct orbits with
\begin{align*}
\Ocl_\alpha(0^+) &= \Ocl_\alpha(0^+, i_2^+) \cdot \Ocl_\alpha(i_2^+, 0^+) \mbox{ and} \\
\Ocl_\alpha(i_1^+) &= \Ocl_\alpha(i_1^+, i_3^+) \cdot \Ocl_\alpha(i_3^+, i_1^+).
\end{align*}
Then we have that $\orb(\alpha_{\gamma_{i_1,i_2,i_3}}) = \orb(\alpha) + 2$
\end{enumerate}
\end{lemma}

To control the orbits of $\alpha$ we will need the following claim:
\begin{claim} \label{claim: properties of alpha gamma}
For $\gamma$ one of $\gamma_{i_1, i_2}$ or $\gamma_{i_1, i_2, i_3}$ as above and $\alpha_\gamma = \sigma \gamma^{-1} \pi \gamma$, we have the following formulas:
\[
\gamma_{i_1,i_2}^{-1} \alpha_{\gamma_{i_1,i_2}} \gamma_{i_1,i_2} (x)
=
\begin{cases}
i_1^+ & \text{if } x = \alpha^{-1}(0^+) \\
i_2^+ & \text{if } x = \alpha^{-1}(i_1^+) \\
0^+ & \text{if } x = \alpha^{-1}(i_2^+) \\
\alpha(x) & \text{else,}
\end{cases}
\]
and
\[
\gamma_{i_1,i_2,i_3}^{-1} \alpha_{\gamma_{i_1,i_2,i_3}} \gamma_{i_1,i_2,i_3} (x)
=
\begin{cases}
i_2^+ & \text{if } x = \alpha^{-1}(0^+) \\
i_3^+ & \text{if } x = \alpha^{-1}(i_1^+) \\
0^+ & \text{if } x = \alpha^{-1}(i_2^+) \\
i_1^+ & \text{if } x = \alpha^{-1}(0^+) \\
\alpha(x) & \text{else.}
\end{cases}
\]
\end{claim}

\begin{proof}[Proof of Claim \ref{claim: properties of alpha gamma}]
We compute 
\begin{align*}
\gamma^{-1} \alpha_{\gamma} \gamma \alpha^{-1}(x) &= \gamma^{-1} \sigma \gamma \pi \gamma^{-1} \gamma \pi^{-1} \sigma^{-1}(x) \\
&= \gamma^{-1} \sigma \gamma \sigma^{-1}(x)
\end{align*}
and further observe that
$\gamma \sigma(x) = \sigma \gamma(x)$ unless 
\begin{itemize}
\item $x \in \sigma^{-1}( \{ 0^+, i_1^+, i_2^+ \} )$ if $\gamma = \gamma_{i_1, i_2}$ or 

\item $x \in \sigma^{-1}( \{ 0^+, i_1^+, i_2^+, i_3^+ \} )$ if $\gamma = \gamma_{i_1, i_2, i_3}$.

\end{itemize}
From this, we see that $\gamma^{-1} \alpha_{\gamma} \gamma \alpha^{-1}(x)=x$ unless 
\begin{itemize}
\item $x \in \{ 0^+, i_1^+, i_2^+ \}$, if $\gamma = \gamma_{i_1, i_2}$ or
\item $x \in \{ 0^+, i_1^+, i_2^+, i_3^+ \}$, if $\gamma = \gamma_{i_1, i_2, i_3}$.
\end{itemize}
A straightforward calculation completes the claim.
\end{proof}

\begin{proof}[Proof of Lemma \ref{lemma:modifying orbits}]
Observe that $\orb(\alpha_{\gamma}) = \orb(\gamma^{-1} \alpha_\gamma \gamma)$, hence it suffices to show the statement for the orbits of $\gamma^{-1} \alpha_\gamma \gamma$.

Suppose that $i_1, i_2$ are as in part (\ref{prop_case:orbits_case1}) of Lemma \ref{lemma:modifying orbits} and set $\gamma = \gamma_{i_1, i_2}$. 
By the previous claim every orbit of $\alpha$ not containing $\{ 0^+, i_1^+, i_2^+ \}$ will be equal to the orbits of $\gamma^{-1} \alpha_\gamma \gamma$.
Furthermore, we see that the orbits $\Ocl_{\gamma^{-1} \alpha_\gamma \gamma}(0^+)$, $\Ocl_{\gamma^{-1} \alpha_\gamma \gamma}(i_1^+)$ and $\Ocl_{\gamma^{-1} \alpha_\gamma \gamma}(i_2^+)$ are distinct. 
Hence indeed
$\orb(\alpha_\gamma) = \orb(\gamma^{-1} \alpha_\gamma \gamma) = \orb(\alpha)+2$.

If $i_1, i_2, i_3$ are as part (\ref{prop_case:orbits_case2}) of Lemma \ref{lemma:modifying orbits}, then
set $\gamma = \gamma_{i_1,i_2,i_3}$ and again observe that every orbit of $\alpha$ not containing $\{ 0^+, i_1^+, i_2^+, i_3^+ \}$ will be equal to the orbits of $\gamma^{-1} \alpha_\gamma \gamma$.
 Moreover, we see that the orbits 
$\Ocl_{\gamma^{-1} \alpha_{\gamma} \gamma}(0^+)$,
 $\Ocl_{\gamma^{-1} \alpha_{\gamma} {\gamma}}(i_1^+)$, $\Ocl_{{\gamma}^{-1} \alpha_{{\gamma}} {\gamma}}(i_2^+)$ and $\Ocl_{{\gamma}^{-1} \alpha_{{\gamma}} {\gamma}}(i_3^+)$ will be distinct
and so again
$\orb(\alpha_{{\gamma}}) =  \orb(\alpha)+2$
\end{proof}

\subsubsection{Proof of part $(i)$ of Theorem \ref{theorem: cl may be used to compute cl}} \label{subsec:proof of exchanging sequences theorem}

We follow the general startegy of \cite{Ivanov-Fialkovski}.
We first need the following claims:

\begin{claim} \label{claim: cl cbi one down}
If $v, w \in \Acl^+$ are related. Then either $\dcbi(v,w) = \cl_{F(\Acl)}(v,w) = 0$ or there is a cyclic block interchange $\tilde{v}$ of $v$ such that $\cl_{F(\Acl)}(\tilde{v} + w^{-1}) = \cl_{F(\Acl)}(v + w^{-1}) - 1$.
\end{claim}

\begin{proof}
Let $v,w \in \Acl^+$ are two related words. 
If $v,w$ are cyclic conjugates of each other then we see that $\cl_{F(\Acl)}(v+w^{-1}) = 0 = \dcbi(v,w)$.
Else, suppose that $v,w$ are related but not cyclic conjugates of each other. Recall that
$$
\cl_{F(\Acl)}(v + w^{-1}) = \min_{\pi \in \Pi_{v+w^{-1}}} \{  \frac{n}{2} - \frac{\orb(\sigma \pi)}{2} \}
$$
where $\Pi_{v+w^{-1}}$ is as in the previous Section. Fix a $\pi \in \Pi_{v + w^{-1}}$ which realizes this minimum.
We define $\Icl_{v+w^{-1}}$, $\sigma \col \Icl_{v+w^{-1}} \to \Icl_{v+w^{-1}}$ and $\alpha = \sigma \pi$ as in Sections \ref{subsec:orbits of sigma pi} and \ref{subsec:modifying the orbits}.
By Lemma \ref{lemma:special points for alpha} there are the following cases:
\begin{enumerate}
\item $\alpha^2(x) = x$ for all $x \in \Icl^+$. Then we see that $v$ and $w$ are cyclic conjugates which is a contradiction. 
\item There are cyclically ordered elements $(i_0, i_1, i_2)$, such that
\[
\Ocl_\alpha(i_0^+) = \Ocl_\alpha(i_0^+, i_2^+) \cdot \Ocl_\alpha(i_2^+, i_1^+) \cdot \Ocl_\alpha(i_1^+, i_0^+).
\]
By cyclically relabeling, which corresponds to taking a cyclic permutation of $v$, we may further assume that $i_0 = 0$.
Define
\begin{align*}
v_1 &= \xt_{0^+} \ldots \xt_{(i_1-1)^+} \\
v_2 &= \xt_{i_1^+} \ldots \xt_{(i_2-1)^+} \\
v_3 &= \emptyset \\
v_4 &= \xt_{i_2^+} \ldots \xt_{(n-1)^+}
\end{align*}
and observe that the function $\gamma_{i_1,i_2} \pi \gamma_{i_1,i_2}^{-1}$ is a pairing for the chain $v_1 v_4 v_2 - w^{-1}$. 
By Lemma \ref{lemma:modifying orbits},
\[
\orb(\sigma \gamma_{i_1,i_2} \pi \gamma_{i_1,i_2}^{-1}) = \orb(\alpha_{\gamma_{i_1,i_2}}) = \orb(\alpha)-2 = \orb(\sigma \pi)-2,
\]
and hence
\begin{eqnarray*}
\cl_{F(\Acl)}(v_1 v_4 v_2 + w^{-1}) &\leq & \frac{n}{2} - \frac{\orb(\sigma \gamma_{i_1,i_2} \pi \gamma_{i_1,i_2}^{-1})}{2} \\
&=& \frac{n}{2} - \frac{\orb(\sigma \pi)}{2} - 1 \\
&=& \cl_{F(\Acl)}(v + w^{-1})-1,
\end{eqnarray*}
since $\pi$ is extremal and the claim follows for $\tilde{v} = v_1 v_4 v_2$.

\item There are cyclically ordered elements $(i_0, i_1, i_2, i_3)$, such that $\Ocl_\alpha(i_0^+)$ and $\Ocl_\alpha(i_1^+)$ are distinct orbits and
\begin{align*}
\Ocl_\alpha(i_0^+) = \Ocl_\alpha(i_0^+, i_2^+) \cdot \Ocl_\alpha(i_2^+, i_0^+) \\
\Ocl_\alpha(i_1^+) = \Ocl_\alpha(i_1^+, i_3^+) \cdot \Ocl_\alpha(i_1^+, i_3^+).
\end{align*}
Again by cyclical relabeling we may assume that $i_0 = 0$.
Similarly as before define
\begin{align*}
v_1 &= \xt_{0^+} \ldots \xt_{(i_1-1)^+} \\
v_2 &= \xt_{i_1^+} \ldots \xt_{(i_2-1)^+} \\
v_3 &= \xt_{i_2^+} \ldots \xt_{(i_3-1)^+} \\
v_4 &= \xt_{i_3^+} \ldots \xt_{(n-1)^+}
\end{align*}
and again observe that the map $\gamma_{i_1,i_2,i_3} \pi \gamma_{i_1,i_2,i_3}^{-1}$ is a pairing for the element $v_1 v_4 v_3 v_2 - w^{-1}$. 
Using Lemma \ref{lemma:modifying orbits} we see that
\[
\orb(\sigma \gamma_{i_1,i_2,i_3} \pi \gamma_{i_1,i_2,i_3}^{-1}) = \orb(\alpha_{\gamma_{i_1,i_2,i_3}}) = \orb(\alpha)-2 = \orb(\sigma \pi)-2,
\]
and hence
\begin{eqnarray*}
\cl_{F(\Acl)}(v_1 v_4 v_3 v_2 + w^{-1}) &\leq & \frac{n}{2} - \frac{\orb(\sigma \gamma_{i_1,i_2,i_3} \pi \gamma_{i_1,i_2,i_3}^{-1})}{2}  \\ 
&=& \frac{n}{2} - \frac{\orb(\sigma \pi)}{2} - 1  \\
&=& \cl_{F(\Acl)}(v+w^{-1})-1
\end{eqnarray*}

since $\pi$ is extremal and the claim follows for $\tilde{v} = v_1 v_4 v_3 v_2$.
\end{enumerate}
This completes the proof of the Claim \ref{claim: cl cbi one down}.
\end{proof}

\begin{claim} \label{claim: a cbi of b then cl 1}
If $a,b \in \Acl^+$ are words such $b$ is a cyclic block interchange of $a$ then $\cl_{F(\Acl)}(a + b^{-1}) \leq 1$.
\end{claim}

\begin{proof}
If $a,b \in \Acl^+$ are as in the claim then let $a'$ be a cyclic permutation of $a$ and let $b'$ be a cyclic permutation of $b$ such that $a' = a_1 a_2 a_3 a_4$ and $b' = a_1 a_4 a_3 a_2$. Then by using basic properties of commutator length (Proposition \ref{prop:formulas for cl}) we may estimate
\begin{eqnarray*}
\cl_{F(\Acl)}(a + b^{-1}) &=& \cl_{F(\Acl)}(a' + b'^{-1}) \\ 
&\leq & \cl_{F(\Acl)}(a_1 a_2 a_3 a_4 a_2^{-1} a_3^{-1} a_4^{-1} a_1^{-1}) \\
&=&  \cl_{F(\Acl)}(a_2 a_3 a_4 a_2^{-1} a_3^{-1} a_4^{-1}) \\
&=& \cl_{F(\Acl)}([a_2 a_3, a_4 a_2^{-1}]) \leq 1.
\end{eqnarray*}
\end{proof}

We may now prove part (i) of Theorem \ref{theorem: cl may be used to compute cl}:

\begin{proof}[Proof of part (i) of Theorem \ref{theorem: cl may be used to compute cl}]
Suppose that $v,w \in \Acl^+$ are two related words.
If $v$ and $w$ are cyclic permutations of each other then $\dcbi(v,w) = 0 = \cl(v + w^{-1})$.
Else, let $z^0, \ldots, z^k$ be a sequence of related words in $\Acl^+$ such that $z^0 = v$, $z^k = w$ and such that $z^i$ is a cyclic block interchange of $z^{i-1}$.
Then
$$
\cl_{F(\Acl)}(v + w^{-1}) \leq \cl_{F(\Acl)}( \sum_{i=1^k} z^{i-1} + (z^i)^{-1}) \leq \sum_{i=1^k} \cl_{F(\Acl)}(z^{i-1} + (z^i)^{-1}) \leq k
$$
where we used Claim \ref{claim: a cbi of b then cl 1}. Thus $\cl_{F(\Acl)}(v + w^{-1}) \leq \dcbi(v,w)$.

On the other hand, if $\cl_{F(\Acl)}(v + w^{-1}) > 0$, then by applying Claim \ref{claim: cl cbi one down} multiple times we obtain a sequence $z^0, \ldots, z^k$ with $k = \cl_{F(\Acl)}(v + w^{-1})$ such that $z^0 = v$, $w = z^k$ and such that $z^i$ is obtained from $z^{i-1}$ by a cyclic block interchange. Thus
$\cl_{F(\Acl)}(v + w^{-1}) \geq \dcbi(v,w)$.
Combining the inequalities we see that $\cl_{F(\Acl)}(v + w^{-1}) = \dcbi(v,w)$.
This finishes the prove of part (i) of Theorem \ref{theorem: cl may be used to compute cl}.
\end{proof}

\subsection{Polynomial reduction of CBI-$\Acl$ to CL-$F(\Acl)$} \label{subsec: poly redcution of cbi to cl}

The aim of this section is to show the ramaining part $(ii)$ of Theorem \ref{theorem: cl may be used to compute cl}:

\begin{manualtheorem}{\ref{theorem: cl may be used to compute cl} (ii)}
There is a polynomial time reduction from CBI-$\Acl$ to CL-$F(\Acl)$.
\end{manualtheorem}

We need the following result:
\begin{claim} \label{claim:reduction to single comm}
Let $v,w \in \Acl^+$ be two related words.
Then
$$
\cl_{F(\Acl)}(v + w^{-1}) = \min \{ \cl_{F(\Acl)}(v \tilde{w}^{-1}) \mid \tilde{w} \mbox{ is a cyclic conjugate of } w \}
$$
\end{claim}

\begin{proof}
From the definition of $\cl_{F(\Acl)}(v + w^{-1})$ it is clear that
$\cl_{F(\Acl)}(v + w^{-1})$ is less or equal than the right-hand side.
Moreover, if $\cl_{F(\Acl)}(v + w^{-1}) = 0$ then $v$ and $w$ are cyclic permutations of each other and thus the statement is clear.

Thus suppose that $\cl_{F(\Acl)}(v + w^{-1}) > 0$.
As before suppose that $v = \xtt_{0^+} \cdots \xtt_{(n-1)^+}$ and $w = \xtt_{0^-} \cdots \xtt_{(n-1)^-}$ and let $\pi$ be a pairing which realizes the commutator length $\cl_{F(\Acl)}(v + w^{-1})$. Let $\sigma \col \Icl_{F(\Acl)} \to \Icl_{F(\Acl)}$ be as above.

As $\cl_{F(\Acl)}(v + w^{-1}) > 0$, we may find an $i^+ \in \Icl^+$ such that $(\sigma \pi)^2(i^+) \not = i^+$.
Without loss of generality we assume that $(n-1)^+$ has this property and that $\pi((n-1)^+) = (n-1)^-$.
Observe that Bardakov's algorithm is also valid if the words are not reduced.
Thus, for $\tilde{w} = w$ we will estimate the commutator length of
$$
z = v \tilde{w}^{-1}.
$$
Note that $\pi$ as above remains a valid pairing for this word.
The map $\sigma$ in Bardakov's algorithm takes the form
$$
\sigma_z (i) =
\begin{cases}
0^- & \mbox{if } i = 0^+ \\
(n-1)^+ & \mbox{if } i = (n-1)^-, \mbox{ and} \\
\sigma(i) & \mbox{else.}
\end{cases}
$$
Then the orbits of $0^+$ of $\alpha$ contain $(n-1)^-$ since $\pi((n-1)^-) = (n-1)^+$.
The orbits of $\alpha_z$ are the same as the orbits of $\alpha$ apart from that this orbit gets split into the orbits $\Ocl_{\alpha_z}(0^+) = \Ocl_{\alpha}(0^+, (n-1)^-)$ and $\Ocl_{\alpha_z}((n-1)^-) = ( (n-1)^-, 0^+)$.
Thus $\orb(\alpha_z) = \orb(\alpha) + 1$. Bardakov yields that
$$
 \cl_{F(\Acl)}(v \tilde{w}) \leq \cl_{F(\Acl)}(v + w^{-1}).
$$
This finishes the claim.
\end{proof}

We may now finish the proof of part $(ii)$ of Theorem \ref{theorem: cl may be used to compute cl}.

\begin{proof}[Proof of part $(ii)$ of Theorem \ref{theorem: cl may be used to compute cl}]

Let $v, w \in \Acl^+$ and $n \in \N$ be the input of the decision problem CBI-$\Acl$. Then we know by part $(i)$ of Theorem \ref{theorem: cl may be used to compute cl} that $\dcbi(v,w) = \cl_{F(\Acl)}(v + w^{-1})$. Claim \ref{claim:reduction to single comm} asserts    that in time $|w|$ we may decide $\cl_{F(\Acl)}(v + w^{-1}) \leq n$, using the decision problem CL-$F(\Acl)$.
This finishes the proof of Theorem \ref{theorem: cl may be used to compute cl}.
\end{proof}

\section{CBI is NP-complete} \label{sec:cbi}

The aim of this section is to show Theorem \ref{theorem:cbi np complete}:

\begin{manualtheorem}{\ref{theorem:cbi np complete}}
Let $\Acl$ be an alphabet with $| \Acl | \geq 2$. Then CBI-$\Acl$ is $\NP$-complete.
\end{manualtheorem}

We will show this using a reduction from $3$-PARTITION to CBI:

\begin{defn}[$3$-PARTITION]
Let $n$ and $N$ be natural numbers where $N$ is uniformly polynomially bounded by $n$. Moreover, let $a_1, \ldots, a_{3n} \in \N$ be integers which satisfy that $N/4 < a_i < N/2$ 
for every $i$ and such that $\sum_{i=1}^{3n} = n N$. 

Then the decision problem which decides if there a partition of $a_1, \ldots, a_{3n}$ into $n$ families consisting of exactly $3$ elements such that the sum of each family is $N$ is called \emph{3-PARTITION}.
The size of the input is polynomial in $n$.
\end{defn}
It is well-known that $3$-PARTITION is $\NP$-complete \cite{comp_intrac}.

This section is organized as follows: In Section \ref{subsec: proof if A has 4 elts} we prove Theorem \ref{theorem:cbi np complete} if $\Acl$ has hat least $4$ elements. In Section \ref{subsec: reduction to binary alphabets} we reduce CBI-$\Acl$ to CBI for binary alphabets, thus proving Theorem \ref{theorem:cbi np complete}.

\subsection{Proof of Theorem \ref{theorem:cbi np complete} if $\Acl$ has $4$ elements} \label{subsec: proof if A has 4 elts}

In this section we will modify the reduction presented in the proof of \cite[Theorem 11]{rev_transp} to reduce CBI-$\Acl$ to $3$-PARTITION if $\Acl = \{ \att, \btt, \ctt, \dtt \}$.

Let $n$, $N$ and $a_1, \ldots, a_{3n}$ be an instance of $3$-PARTITION.
Analogously to the proof of \cite[Theorem 11]{rev_transp} we define words $v$ and $w$ via
\begin{eqnarray*}
v &=&  \att^{n+1} (\btt \ctt^{a_1} \dtt) (\btt \ctt^{a_2} \dtt) \cdots (\btt \ctt^{a_{3n}} \dtt) \btt \text{, and}\\
w &=& (\att \ctt^N \dtt^3)^n \att \btt^{3n+1}.
\end{eqnarray*}
Observe that $v$ and $w$ are related.

\begin{prop}
The above instance of $3$-PARTITION is solvable if and only if $\dcbi(v,w) = 3n$.
\end{prop}

\begin{claim}
If this instance of $3$-PARTITION is solvable, then $\dcbi(v,w) \leq 3 n$.
\end{claim}

\begin{proof}
If this instance of $3$-PARTITION is solvable, then we may move the $3n$ many subwords of the form $\ctt^{a_i} \dtt$ within the word to the $\att$'s. Thus, in this case,
 $\dcbi(v,w) \leq 3n$.
\end{proof}

\begin{claim} \label{claim:if cbi 3n then 3 part solvable}
If $\dcbi(v,w) \leq 3n$ then in fact $\dcbi(v,w) = 3n$ and the instance of $3$-PARTITION is solvable.
\end{claim}

To show this claim we will define an auxiliary function $\nu \col \Acl^+ \to \Z$ as follows:
 Let $\xtt, \ytt \in \Acl$ be two letters.
For a word $w \in \Acl^+$ we define $\nu_{\xtt \ytt}$ to be the number of times $\xtt \ytt$ is  a cyclic subword  of $w$. I.e.\ $\nu_{\xtt \ytt}(w)$ counts the number of occurrences of $\xtt \ytt$ as subwords in $w$ and it counts if the last letter of $w$ is $\xtt$ and the first letter of $w$ is $\ytt$. Observe that then $\nu_{\xtt \ytt}(w)$ is invariant under a cyclic permutation of the letters of $w$. 
We define define the function $\nu \col \Acl^+ \to \Z$ as
$$
\nu = \nu_{\att \att} + \nu_{\btt \ctt} + \nu_{\ctt \dtt} + \nu_{\dtt \btt} - \nu_{\att \ctt} - \nu_{\ctt \ctt} - \nu_{\dtt \dtt} - \nu_{\dtt \att} - \nu_{\btt \btt}.
$$
The definition of $\nu_{xy}$ is analogous to the homogeneous Brooks quasimorphisms \cite{brooks}.

\begin{lemma} \label{lemma: nu}
The function $\nu \col \Acl^+ \to \Z$ has the following properties:
\begin{enumerate}
\item \label{item:values of nu} 
$\nu(v) - \nu(w) = 18 n$ and
\item \label{item:computer part less than 6} If $x,y \in \Acl^+$ are such that $x$ is obtained from $y$ by a cyclic block interchange then $\nu(x)-\nu(y) \leq 6$.
\item \label{item: no cc occurences broken}
If $y$ is obtained from $x$ by a block interchange such that $\nu(x) - \nu(y) = 6$. Then the block interchange does not cut a consecutive sequence of $\ctt$.
\end{enumerate}
\end{lemma}

\begin{proof}
To show item (\ref{item:values of nu}), observe that
$$
\begin{array}{lll}
\nu_{\att \att}(v) = n, & \nu_{\btt \ctt}(v)=3n, & \nu_{\ctt \dtt}(v) = 3 n, \\
\nu_{\dtt \btt}(v) = 3 n, & \nu_{\att \ctt}(v) = 0, & \nu_{\ctt \ctt}(v) = n N -3n, \\ 
\nu_{\dtt \dtt}(v) = 0, & \nu_{\dtt \att}(v) = 0, & \nu_{\btt \btt}(v) = 0,
\end{array}
$$
and hence $\nu(v) = 13 n - n N$. Similarly we compute
$$
\begin{array}{lll}
\nu_{\att \att}(w) = 0, & \nu_{\btt \ctt}(w)=0, & \nu_{\ctt \dtt}(w) = n, \\
\nu_{\dtt \btt}(w) = 0, & \nu_{\att \ctt}(w) = n, & \nu_{\ctt \ctt}(v) = n N -n, \\ 
\nu_{\dtt \dtt}(v) = 2 n, & \nu_{\dtt \att}(w) = n, & \nu_{\btt \btt}(w) = 3n,
\end{array}
$$
and thus $\nu(w) = - 5 n - n N$.
Combining both we see that $\nu(v) - \nu(w) = 18 n$.

For item (\ref{item:computer part less than 6}) observe that all of the $\nu_{\xtt \ytt}$ are invariant under cyclic permutations of the letters. Thus we may assume  that $x = x_1 x_2 x_3 x_4$ and $y = x_1 x_4 x_3 x_2$.
Then we see that for any $\nu_{\xtt \ytt}$ the value of
$\nu_{\xtt \ytt}(x) - \nu_{\xtt \ytt}(y)$ does just depend on the first and last letter of the $x_i$. Thus we may assume that each of the $x_i$ have wordlength at most $2$.

If one of the $x_i$ are empty then the block interchange only permutes three words. After cyclically permuting $y$ we may assume that $x = z_1 z_2 z_3$ and $y = z_1 z_3 z_2$ for some $z_1, z_2, z_3 \in \Acl$.
As there are just three breakpoints of $x$ and $y$ we see that $\nu(x)-\nu(y) \leq 6$ and if $\nu(x) - \nu(y) = 6$ then no $\ctt \ctt$ gets separated.

If none of the $x_1, x_2, x_3, x_4$ are empty, then we may assume that they all have word-length exactly $2$, since if any of the $x_i$ is a single letter $\xtt$ we may replace it by $\xtt \xtt$ to get the same result. Thus we may assume that all of the $x_i$ consist of exactly two letters.
For each $x_i$ there are thus $16$ possibilities.
This yields a total of $16^4 = 65536$ possibilities for $x_1, \ldots, x_4$ and in every case we may check that $\nu(x)-\nu(y) \leq 6$ and that $\nu(x)-\nu(y)=6$ only if no sequence of $\ctt \ctt$ gets broken. A MATLAB code which verifies this statement may be found in the appendix.
This finished the proof of the lemma.
\end{proof}

We are now ready to prove Claim \ref{claim:if cbi 3n then 3 part solvable}:
\begin{proof}[Proof of Claim \ref{claim:if cbi 3n then 3 part solvable}]
Suppose that  $v$, $w$ are as above   and that $\dcbi(v,w) \leq 3 n$.
Then there is a sequence $z^0, \ldots, z^m$ with $v = z^0$, $z^{m} = w$, $m = \dcbi(v,w)$ and such that $z^i$ is obtained from $z^{i-1}$ by a cyclic block interchange for all $i \in \{1, \ldots, m \}$.
By Lemma \ref{lemma: nu} we see that $\nu(z^i)-\nu(z^{i+1}) \leq 6$. Thus
$$
18 n = \nu(v) - \nu(w) = \sum_{i=0}^{m-1} \nu(z^i) - \nu(z^{i+1}) \leq 6 m
$$
and hence $3 n \leq \dbi^\circ(v,w)$ but by the assumptions of the claim we conclude that $3 n = \dcbi(v,w)$ and that $\nu(z^{i}) - \nu(z^{i+1}) = 6$ for all $i \in \{ 1, \ldots, m \}$.
But by Lemma \ref{lemma: nu} this means that $w$ is obtained from $v$ without cutting any $\ctt \ctt$ sequence.
Thus we obtain a valid solution for $3$-PARTITION.
\end{proof}

\subsection{Reduction to binary alphabets} \label{subsec: reduction to binary alphabets}

The aim of this section is to show Theorem \ref{theorem:cbi np complete} for binary alphabets.
We will use the previous section and show how we may compute $\dcbi$ for alphabets in four letters also in binary alphabets.
For what follows let $\Acl = \{ \att, \btt, \ctt, \dtt \}$ and let $\Bcl = \{ \xtt, \ytt \}$.
Following \cite[Theorem 9]{rev_transp} We will define the map
$$
\lambda \col \Acl^+ \to \Bcl^+
$$
as follows.
Given a positive word $v \in \Acl^+$ with $v = \vtt_1 \cdots \vtt_n$ with $\vtt_i \in \Acl$, we set 
$$
\lambda(v) = (\xtt \ytt^{\epsilon_{\vtt_1}} \xtt)^{4 n^2 + 1} \cdots (\xtt \ytt^{\epsilon_{\vtt_n}} \xtt)^{4 n^2 +1 }
$$
where
$\epsilon_\att = 2$, $\epsilon_\btt = 3$, $\epsilon_\ctt = 4$ and $\epsilon_\dtt = 5$.

\begin{lemma} \label{lemma: binary alphabets}
If $v,w \in \Acl^+$ are related, then $\lambda(v), \lambda(w) \in \Bcl^+$ are related and
$$
\dcbi(\lambda(v), \lambda(w)) = \dcbi(v,w).
$$
\end{lemma}

\begin{proof}

It is easy to see that if $v' \in \Acl^+$ is a cyclic block transformation of $v$ then $\lambda(v')$ is a cyclic block transformation of $v$.
Thus we see that
$$
\dcbi(\lambda(v), \lambda(w)) \leq \dcbi(v,w).
$$
We want to show the other direction. 
Fix a sequence $z^0, \ldots, z^t$, with $z^0 = \lambda(v)$, $z^t = \lambda(w)$, such that each $z^i$ is obtained from $z^{i-1}$ by a cyclic block interchange and where $\dcbi(\lambda(v), \lambda(w)) = t$. Note that by the above inequality we may crudely estimate $t \leq n$.

Suppose that $v = \vtt_1 \cdots \vtt_n$ and $w = \wtt_1 \cdots \wtt_n$. 
Every cyclic block transformations cuts the words $\lambda(z^i)$ in at most $4$ places. Thus there are at most $4n$ cuts from $\lambda(v)$ to $\lambda(w)$.
For every index $i \in \{ 1, \ldots, n \}$ let $\Vcl_i$ be the family of subwords $(\xtt \ytt^{\epsilon_{\vtt_i}} \xtt)$ which do not get cut by the cyclic block interchange, and similarly, let $\Wcl_i$ be the family of subwords for $(\xtt \ytt^{\epsilon_{\wtt_i}} \xtt)$ of $\lambda(w)$ which did not get cut in the sequence $z^0, \ldots, z^t$.
We note that we have that $4 n^2 + 1 - 4 n \leq |\Vcl_i | \leq 4 n^2 + 1$ and similarly $4 n^2 + 1 - 4n \leq | \Wcl_i | \leq 4 n^2 + 1$ for all $i \in \{1, \ldots, n \}$.
This also shows that all of the $\Vcl_i$ are non-empty as $|\Vcl_i| > 4 n^2 + 1 - 4n = (2 n^2 - 1)^2$ and $n \geq 1$. 
Moreover, by recording the cyclic block interchanges we get a bijection
$$
\rho \col \Vcl \to \Wcl
$$
where $\Vcl = \cup_{i=1}^n \Vcl_i$ and $\Wcl = \cup_{i=1}^n \Wcl_i$.
This bijection has the following property:

\begin{claim}
There is a choice of elements $v_i \in \Vcl_i$ for every $i \in \{1, \ldots, n \}$ such that $\rho(v_i) \in \Wcl_{\pi(i)}$, where $\pi$ is a permutation of $\{1, \ldots, n \}$
\end{claim}

\begin{proof}
Observe that the sets $\Vcl_1, \ldots, \Vcl_n$, $\Wcl_1, \ldots, \Wcl_n$ and the map $\rho$ have the following property:
There are no subsets $S, T \subset \{ 1, \ldots, n \}$ with $|S|>|T|$ such that
$$
\rho(\cup_{i \in S} \Vcl_i) \subset \cup_{j \in T} \Wcl_j.
$$
To see this suppose that $S$ and $T$ fulfill the above inclusion. Then we estimate
\begin{eqnarray*}
|\rho(\cup_{i \in S} \Vcl_i)| & \geq & |S| (4 n^2 + 1 - 4 n) \mbox{ and} \\
| \cup_{j \in T} \Wcl_j| & \leq & |T| (4 n^2 + 1) \leq (|S|-1)(4 n^2 + 1)
\end{eqnarray*}
using that $|T| \leq |S| - 1$. Thus
\begin{eqnarray*}
(|S|-1)(4 n^2 + 1) & \geq & |S| (4 n^2 + 1 - 4 n) \\
 - 4 n^2  - 1  & \geq & - 4 n |S|
\end{eqnarray*}
which is a contradiction as $|S| \leq n$.
Thus the collection of sets $\Vcl$ satisfies the marriage condition of Hall's marriage theorem \cite{Hall}.

\end{proof}

Thus, we may focus on the elements $v_i$. The block permutations do not affect those subwords. Thus, every sequence of block permutations to the $v_i$ may be done to the original word $v$. This shows that $\dcbi(v,w) = t$, which is a contradiction.
\end{proof}

We may now prove Theorem \ref{theorem:cbi np complete}:

\begin{proof}[Proof of Theorem \ref{theorem:cbi np complete}]
Let $\Acl$ be an alphabet.
We may assume that $\Acl$ just has two letters, else we may restrict it accordingly.
We know that CBI is NP hard if $\Acl$ has four vertices. Using Lemma \ref{lemma: binary alphabets} we may reduce this in polynomial time to the case where $\Acl$ is a binary alphabet.
\end{proof}

\section{Proof of Theorem \ref{theorem:cl np complete}}
\label{sec:proof of thm a}

By combining all the results of this paper we may prove Theorem \ref{theorem:cl np complete}:

\begin{proof}[Proof of Theorem \ref{theorem:cl np complete}]
Let $F(\Acl)$ be the a non-abelian free group on some alphabet $\Acl$. Thus, $|\Acl| \geq 2$.
By Corollary \ref{corr:cl is in np} CL-$F(\Acl)$ is in $\NP$. Combining Theorem \ref{theorem:cbi np complete} and Theorem \ref{theorem: cl may be used to compute cl} (ii) we see that CL-$F(\Acl)$ is $\NP$-hard. We conclude that CL-$F(\Acl)$ is $\NP$-complete.

If, given an element $g \in F(\Acl)$, there is an algorithm which computes $\cl_{F(\Acl}(g)$ in polynomial time, then observe that $\cl_{F(\Acl)} \leq |g|$, by Bardakov's algorithm. Thus we would be able to decide CL-$F(\Acl)$ in polynomial time and thus P=NP. By Proposition \ref{prop:formulas for cl} the same holds if $G$ has a free retract.
\end{proof}

\section{Appendix: MATLAB program for Lemma \ref{lemma: nu}} \label{sec:appendix}
Here we give the codes to prove the rest of items (\ref{item:computer part less than 6}) and (\ref{item: no cc occurences broken}) of Lemma \ref{lemma: nu}.
First we give an implementation of $\nu_{\xtt \ytt}$:

\vspace{10pt}
\begin{verbatim}
function sol = nu(xy,v)
%%% compute \nu_{xy}(v)
count = 0;
for i=1:(length(v)-1)
   if v(i)==xy(1) && v(i+1)==xy(2)
       count = count + 1;
   end
end
%% check if there is an occurence of xy in the first and last letter of v:
if v(end)== xy(1) && v(1) == xy(2)
    count = count + 1;
end
sol = count;
end
\end{verbatim}
\vspace{10pt}

Now we use this to define $\nu$:

\vspace{10pt}
\begin{verbatim}
function sol = nu_total(v)
%%% computes $\nu(v)$:
sol = nu('aa',v) + nu('bc',v) + nu('cd',v) + nu('db',v) - nu('ac',v) - ...
 nu('cc',v) - nu('dd',v) - nu('da',v) - nu('bb',v);
end
\end{verbatim}
\vspace{10pt}

We can now check the assertions of the Lemma. To show the rest of item (\ref{item:computer part less than 6}), we check that
$\nu(x_1 x_2 x_3 x_4) - \nu(x_1 x_4 x_3 x_2) \leq 6$ for all two letter words $x_i$ in the alphabet $\Acl = \{ \att, \btt, \ctt, \dtt \}$.

\vspace{10pt}
\begin{verbatim}
function check = nu_less_than_6
check = true;
A = ['a','b','c','d'];
for x_11 = A
for x_12 = A
for x_21 = A
for x_22 = A
for x_31 = A
for x_32 = A
for x_41 = A
for x_42 = A
    x_1 = [x_11,x_12];
    x_2 = [x_21,x_22];
    x_3 = [x_31,x_32];
    x_4 = [x_41,x_42];
    nu_temp = nu_total([x_1,x_2,x_3,x_4])-nu_total([x_1,x_4,x_3,x_2]);
    %%% check if nu_temp is strictly larger than 6:
    if nu_temp >= 7
        check = false;
    end
end
end
end
end
end
end
end
end
end
\end{verbatim}
\vspace{10pt}
This returns \texttt{true}.

To check item (\ref{item: no cc occurences broken}), we use the following code:

\vspace{10pt}
\begin{verbatim}
function check = nu_6_no_c_broken
check = true;
A = ['a','b','c','d'];
for x_11 = A
for x_12 = A
for x_21 = A
for x_22 = A
for x_31 = A
for x_32 = A
for x_41 = A
for x_42 = A
    x_1 = [x_11,x_12];
    x_2 = [x_21,x_22];
    x_3 = [x_31,x_32];
    x_4 = [x_41,x_42];
    nu_temp = nu_total([x_1,x_2,x_3,x_4])-nu_total([x_1,x_4,x_3,x_2]);
    if nu_temp == 6
        %%check if any of the letters between x_1 and x_2, x_2 and x_3, x_3 
        %%and x_4, and x_4 and x_1 are c's:
        if x_12 == 'c' && x_21 == 'c'
            check = false;
        elseif x_22 == 'c' && x_31 == 'c'
            check = false;
        elseif x_32 == 'c' && x_41 == 'c'
            check = false;
        elseif x_42 == 'c' && x_11 == 'c'
            check = false;
        end
    end
end
end
end
end
end
end
end
end
end
\end{verbatim}
\vspace{10pt}
This also returns \texttt{true}. This finishes the proof of Lemma \ref{lemma: nu}.

{\small
\bibliographystyle{alpha}
\bibliography{complex_cl}}

\newcommand{\etalchar}[1]{$^{#1}$}
\begin{thebibliography}{{Heu}19b}

\bibitem[Bar00]{Bardakov}
V.~G. Bardakov.
\newblock Computation of commutator length in free groups.
\newblock {\em Algebra Log.}, 39(4):395--440, 507--508, 2000.

\bibitem[Bra12]{Lukas}
Lukas Brantner.
\newblock On the complexity of sails.
\newblock {\em Pacific J. Math.}, 258(1):1--30, 2012.
\newblock With an appendix by Frederick Manners.

\bibitem[Bro81]{brooks}
Robert Brooks.
\newblock Some remarks on bounded cohomology.
\newblock In {\em Riemann surfaces and related topics: {P}roceedings of the
  1978 {S}tony {B}rook {C}onference ({S}tate {U}niv. {N}ew {Y}ork, {S}tony
  {B}rook, {N}.{Y}., 1978)}, volume~97 of {\em Ann. of Math. Stud.}, pages
  53--63. Princeton Univ. Press, Princeton, N.J., 1981.

\bibitem[Cal09a]{Calegari}
Danny Calegari.
\newblock {\em scl}, volume~20 of {\em MSJ Memoirs}.
\newblock Mathematical Society of Japan, Tokyo, 2009.

\bibitem[Cal09b]{scl_rational}
Danny Calegari.
\newblock Stable commutator length is rational in free groups.
\newblock {\em J. Amer. Math. Soc.}, 22(4):941--961, 2009.

\bibitem[CH19]{heuer_chen_spectral}
Lvzhou {Chen} and Nicolaus {Heuer}.
\newblock {Spectral gap of scl in graphs of groups and $3$-manifolds}.
\newblock {\em arXiv e-prints}, page arXiv:1910.14146, Oct 2019.

\bibitem[Chr98]{Christie}
David~Alan Christie.
\newblock Genome rearrangement problems ordering problem.
\newblock {\em PhD thesis, University of Glasgow.}, 1998.

\bibitem[Cul81]{Culler}
Marc Culler.
\newblock Using surfaces to solve equations in free groups.
\newblock {\em Topology}, 20(2):133--145, 1981.

\bibitem[DH91]{DH}
Andrew~J. Duncan and James Howie.
\newblock The genus problem for one-relator products of locally indicable
  groups.
\newblock {\em Math. Z.}, 208(2):225--237, 1991.

\bibitem[FI15]{Ivanov-Fialkovski}
D.~{Fialkovski} and S.~O. {Ivanov}.
\newblock {Minimal commutator presentations in free groups}.
\newblock {\em ArXiv e-prints}, April 2015.

\bibitem[FLR{\etalchar{+}}09]{CGR}
Guillaume Fertin, Anthony Labarre, Irena Rusu, \'{E}ric Tannier, and
  St\'{e}phane Vialette.
\newblock {\em Combinatorics of genome rearrangements}.
\newblock Computational Molecular Biology. MIT Press, Cambridge, MA, 2009.

\bibitem[GJ79]{comp_intrac}
Michael~R. Garey and David~S. Johnson.
\newblock {\em Computers and Intractability: A Guide to the Theory of
  NP-Completeness}.
\newblock W. H. Freeman \& Co., USA, 1979.

\bibitem[GT79]{GT}
Richard~Z. Goldstein and Edward~C. Turner.
\newblock Applications of topological graph theory to group theory.
\newblock {\em Math. Z.}, 165(1):1--10, 1979.

\bibitem[Hal09]{Hall}
Philip Hall.
\newblock On representatives of subsets.
\newblock In {\em Classic Papers in Combinatorics}, pages 58--62. Springer,
  2009.

\bibitem[Heu19a]{heuer_raags}
Nicolaus Heuer.
\newblock Gaps in {SCL} for amalgamated free products and {RAAG}s.
\newblock {\em Geom. Funct. Anal.}, 29(1):198--237, 2019.

\bibitem[{Heu}19b]{heuer_scl_rp}
Nicolaus {Heuer}.
\newblock {The full spectrum of scl on recursively presented groups}.
\newblock {\em arXiv e-prints}, page arXiv:1909.01309, Sep 2019.

\bibitem[HL19]{heuer_scl_simvol}
Nicolaus {Heuer} and Clara {L{\"o}h}.
\newblock {The spectrum of simplicial volume}.
\newblock {\em arXiv e-prints}, page arXiv:1904.04539, Apr 2019.

\bibitem[KLMT10]{SolvProblemFreeGroup}
O.~Kharlampovich, I.~G. Lys\"{e}nok, A.~G. Myasnikov, and N.~W.~M. Touikan.
\newblock The solvability problem for quadratic equations over free groups is
  {NP}-complete.
\newblock {\em Theory Comput. Syst.}, 47(1):250--258, 2010.

\bibitem[RSW05]{rev_transp}
Andrew Radcliffe, A.~Scott, and Elizabeth Wilmer.
\newblock Reversals and transpositions over finite alphabets.
\newblock {\em SIAM Journal on Discrete Mathematics}, 19, 01 2005.

\end{thebibliography}

\vfill

\noindent
\emph{Nicolaus Heuer}\\[.5em]
  {\small
  \begin{tabular}{@{\qquad}l}
DPMMS,    University of Cambridge \\
    \textsf{nh441@cam.ac.uk},
    \textsf{https://www.dpmms.cam.ac.uk/$\sim$nh441}
  \end{tabular}}

\end{document}